\documentclass{amsart}

\usepackage{amsmath}
\usepackage{amssymb}
\usepackage{amsfonts}
\usepackage[pdftex]{graphicx}

\theoremstyle{plain}

\newtheorem{theorem}{Theorem}[section]
\newtheorem{lemma}[theorem]{Lemma}
\newtheorem{proposition}[theorem]{Proposition}
\newtheorem{corollary}[theorem]{Corollary}
\newtheorem{example}{Example}

\newtheorem{definition}{Definition}

\begin{document}

\title[The Spiral Index of Knots] {The Spiral Index of Knots}

\date{\today}

\author[Adams]{Colin Adams}
\author[George]{William George}
\author[Hudson]{Rachel Hudson}
\author[Morrison]{Ralph Morrison}
\author[Starkston]{Laura Starkston}
\author[Taylor]{Samuel Taylor}
\author[Turanova]{Olga Turanova}

\address{Colin Adams, Department of Mathematics, Bronfman Science Center, Williams College, Williamstown, MA 01267}
\email{Colin.C.Adams@williams.edu}
\address{William George, Department of Mathematics, University of Toronto, Toronto, Ontario, Canada M5S 2E4}
\email{wgeorge@math.toronto.edu}
\address{Rachel Hudson, Department of Mathematics, Bronfman Science Center,Williams College, Williamstown, MA 01267}
\email{10rah@williams.edu}
\address{Ralph Morrison, Department of Mathematics, Bronfman Science Center,Williams College , Williamstown, MA 01267}\email{10rem@williams.edu}
\address{Laura Starkston, Department of Mathematics, Harvard University, 1 Oxford St., Cambridge, MA 02138}\email{lstarkst@fas.harvard.edu}
\address{Samuel Taylor, Mathematics Department, College of New Jersey, Ewing, NJ 08628-0718}
\email{taylor29@tcnj.edu}
\address{Olga Turanova, Department of Mathematics, Barnard College, Columbia University, New York, NY 10027-6598}
\email{ot2130@barnard.edu}
\maketitle

\begin{abstract}
In this paper, we introduce two new invariants that are closely related to Milnor's curvature-torsion invariant. The first, the spiral index of a knot, captures the number of maxima in a knot projection that is free of inflection points. This invariant is closely related to both the bridge and braid index of the knot. The second invariant, the projective superbridge index, provides a method of counting the greatest number of maxima that occur in a given knot projection. In addition to investigating the relationships among these invariants, we classify all knots that satisfy $(\kappa+\tau)[K] = 6\pi$.
\end{abstract}

\section{Introduction}
A projection of a knot $K$ is said to be in braid form if as we traverse the knot in its projection plane we wind monotonically about a fixed axis. The braid index of the knot, denoted $\beta[K]$,  is then defined to be the minimum number of times the knot winds about the fixed axis in any of its braid forms. This invariant has shown to be a powerful tool for knot theorists because of its close relationship to other classical invariants like bridge index and newer invariants like the HOMFLYPT polynomial. In this paper we investigate knot projections that are similar to braid form but without the constraint of a fixed axis. Here we require the knot projection to wind in a single direction; that is, we do not allow the planar curvature of the knot projection to change sign. We say that this type of knot projection is in 
spiral form and call the minimum number of ``spirals'' necessary for such a projection the spiral index of the knot. 

Spiral index, like braid index, is closely related to other properties of the knot. For instance, both impose significant structure on the Seifert circles generated from a diagram of the knot. We also relate the spiral index of a knot to its bridge index, denoted $b[K]$,  by showing that a projection of a knot in spiral form has the same number of maxima in all directions in its plane of projection. To study this more precisely, we introduce the projective superbridge index of a knot, which is related to the superbridge index defined in \cite{Kuiper}, and which captures the greatest number of maxima of a knot in a projection, minimized over all projections. By relating the spiral index and projective superbridge number to the curvature-torsion invariant defined by Milnor in 1953, we are able to identity all knots for which curvature-torsion is equal to $6\pi$, a question that was posed by Honma and Saeki \cite{Honma}.

\section{Spiral Index}
We begin by introducing the notion of spiral form for a knot. Let $K$ be a knot in $\mathbb{R}^3$ and, given $v\in S^2$, let $P_{v}(K)$ be the projection of $K$ onto the plane $P$ through the origin and orthogonal to $v$. When the vector $v$ is clear from context, we denote the projection $P(K)$. In a projection, an inflection point is a point at which planar curvature is zero. We further classify inflection points as either $s$-inflection points, where planar curvature switches signs, or $u$-inflection points, where sign is unchanged. 

\begin{definition}
A smooth projection of a knot $K$ is in \textnormal{spiral form} if the projection has no $s$-inflection points.
\end{definition}

Note that this definition is equivalent to the curvature of the projection being non-negative at all points or non-positive at all points. Further, a projection is in spiral form if and only if it has the same number of local maxima with respect to every direction in the plane (see section 4 for a proof of the equivalence). 

\begin{definition}
The number of maxima in a projection $P(K)$ that is in spiral form is the \textnormal{spiral number} of that projection, $sp(P(K))$. The \textnormal{spiral index} of a knot $K$, denoted by $sp[K]$, is the minimum of $sp(P(K))$ over all spiral projections of $K$. 
\end{definition}

Spiral form is closely related to the braid index of a knot, denoted $\beta[K]$. In fact, because every knot has a projection in braid form \cite{Alexander}, every knot also has a projection in spiral form -- given a braid projection, we may isotope it to be sufficiently close to a circle centered at the braid axis and thus to have no inflection points. It follows immediately that $sp[K] \leq \beta[K]$. Further, since the minimum number of maxima over all projections of a knot is no greater than the minimum over all spiral projections, we have $b[K]\le sp[K]$.

Even though the braid index of a knot may be defined as the minimal number of maxima of a knot projection that winds monotonically around a fixed axis, often the braid index and spiral index of a knot differ. This is because in spiral form, we do not require a fixed axis. For instance, the projection of $6_1$ in Figure \ref{61spiral} has spiral number 3, while the braid index of $6_1$ is $4$.  Thus, we have $sp[6_1]<\beta [6_1]$.
 
\begin{definition}
A knot $K$ is a \textnormal{curly knot} if $sp[K]<\beta [K]$.   More specifically, we say $K$ is an \textnormal{$n$-curly knot} if it is a curly knot with $sp[K]=n$.
\end{definition}

\begin{figure}
\includegraphics[scale =.8]{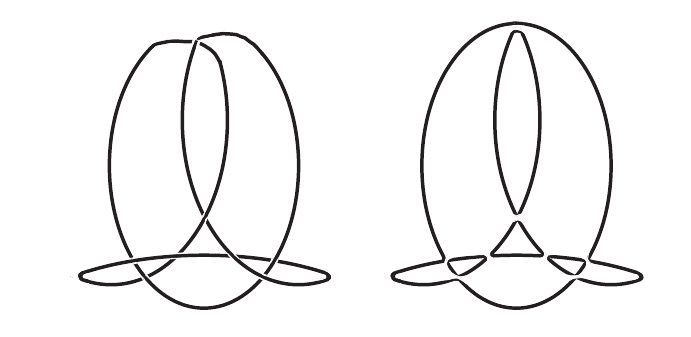}
\caption{The Spiral form of $6_1$ and its Seifert Diagram}
\label{61spiral}
\end{figure}

Curly knots are relatively common among low crossing knots.  Of the 81 nontrivial prime knots up to 9 crossings, there are 26 that we know to be curly knots and only 3 which may or may not be curly. We utilize  Theorem \ref{altprime} and its following paragraph to compute that the $3$-curly knots up to $9$ crossings are $6_1$, $7_2$, $7_6$, $8_4$, $8_6$, $8_8$, $8_{15}$, $9_4$, $9_7$, $9_{11}$, $9_{20}$, $9_{24}$, and $9_{28}$. The known $4$-curly knots up to $9$ crossings are $8_1$, $8_3$, $8_{12}$, $9_2$, $9_8$, $9_{12}$, $9_{14}$, $9_{15}$, $9_{19}$, $9_{21}$, $9_{25}$, $9_{39}$, and $9_{41}$. It is possible that $9_5$, $9_{35}$, and $9_{37}$ are $4$-curly knots but we suspect that their spiral numbers are $5$, which would imply they are not curly. In order to further characterize curly knots, we must consider the Seifert circles generated by  spiral projections.

Spiral form imposes significant structure on the corresponding diagram of Seifert circles. The minimal spiral form of $6_1$ has five Seifert circles as in Figure \ref{61spiral}; four of these circles ``bulge out'' on all sides, and one circle ``bulges in'' on all sides. We formalize these notions as follows. 

Given a Seifert circle $C$, connect its adjacent vertices by line segments (See Figure \ref{bulges}). Call the resulting graph $P$. If $C$ is contained in $P$, we call $C$ an \emph{i-circle} (the sides of $C$ ``bulge in'').  If  $C$ contains $P$, we call $C$ an \emph{o-circle} (the sides of $C$ ``bulge out''). This characterization is equivalent to saying that for an o-circle all sides curve towards the center, while for i-circles all sides curve away from it.

\begin{figure}
\includegraphics[scale=.4]{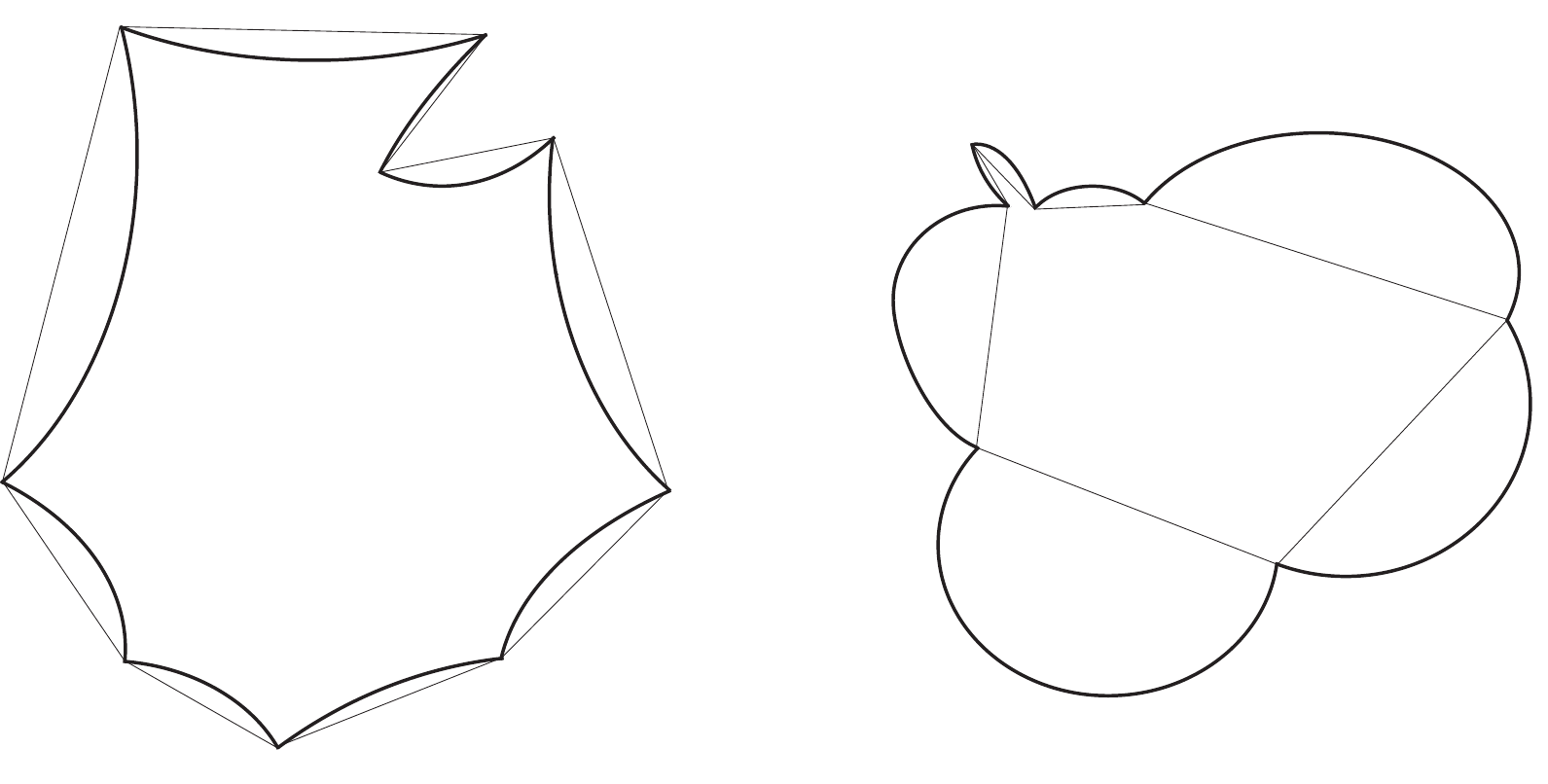}
\caption{i-circles versus o-circles}
\label{bulges}
\end{figure}
 
\begin{lemma} 
For a knot projection in spiral form, any Seifert circle of a Seifert diagram is either an i-circle or an o-circle.
\end{lemma}

\begin{proof} 
Since each edge of the Seifert circle between two vertices has curvature of unchanging sign, we can define each side as curving inward or curving outward. Therefore if there were a Seifert circle which is neither an i-circle nor an o-circle, there would be a vertex at which one edge curving inward met another edge curving outward. Taking into account the orientation of the circle, this implies that the curvatures of these two edges have opposite signs. Since the orientation of the Seifert circle is the same as the orientation of the knot, this contradicts the assumption that the projection is in spiral form.
\end{proof}

\begin{lemma}  The total planar curvature, $\kappa$, of a knot in spiral form with spiral number $sp(P(K))$ is $2\pi sp(P(K))$.
\end{lemma}
\begin{proof}  Let $\mu(K,v)$ denote the number of maxima of the function $K(t)\cdot v$ where $v\in S^1$. Milnor proves in \cite{Milnorcurves} that the total absolute curvature of a knot projection equals $2\pi$ times the average of $\mu(K,v)$ over the unit circle.  Since in the spiral projection $\mu(K,v)=sp(P(K))$ for all $v\in S^1$, it follows that $\kappa=2\pi sp(P(K))$.
\end{proof}

In analogy to the characterization of the braid index as the minimum number of Seifert circles in any projection of the knot(c.f. \cite{Yamada}), we have the following theorem.

\begin{theorem}
The spiral number of a diagram in spiral form is equal to the number of o-circles minus the number of i-circles.
\end{theorem}
\begin{proof}
 Consider the Seifert circles of an arbitrary projection of our knot $K$ in spiral form. Let $i$ be the number of i-circles and $o$ be the number of o-circles. Each Seifert circle, whether it is an i- or an o-circle, may have vertices of one or both of the two possible forms shown in Figure \ref{exteriorangle}. We will distinguish convex cusps from nonconvex cusps and refer to these as $C$ cusps and $N$ cusps respectively. Each crossing of the knot is a vertex of two distinct Seifert circles. One of these vertices must be an $C$ cusp and the other a $N$ cusp, as illustrated in Figure \ref{exteriorangle}. Further, there is an angle between the tangents of the two strands of the knot at the $i^{th}$ crossing. We call this the \emph{exterior angle} of the crossing, and denote it by $\theta^i$. When we refer to the exterior angle as being associated with an $C$ cusp, we will denote it $\theta_C^i$ and similarly, $\theta_N^i$ for $N$ cusps.
 
\begin{figure}[here]
\includegraphics[scale=.5]{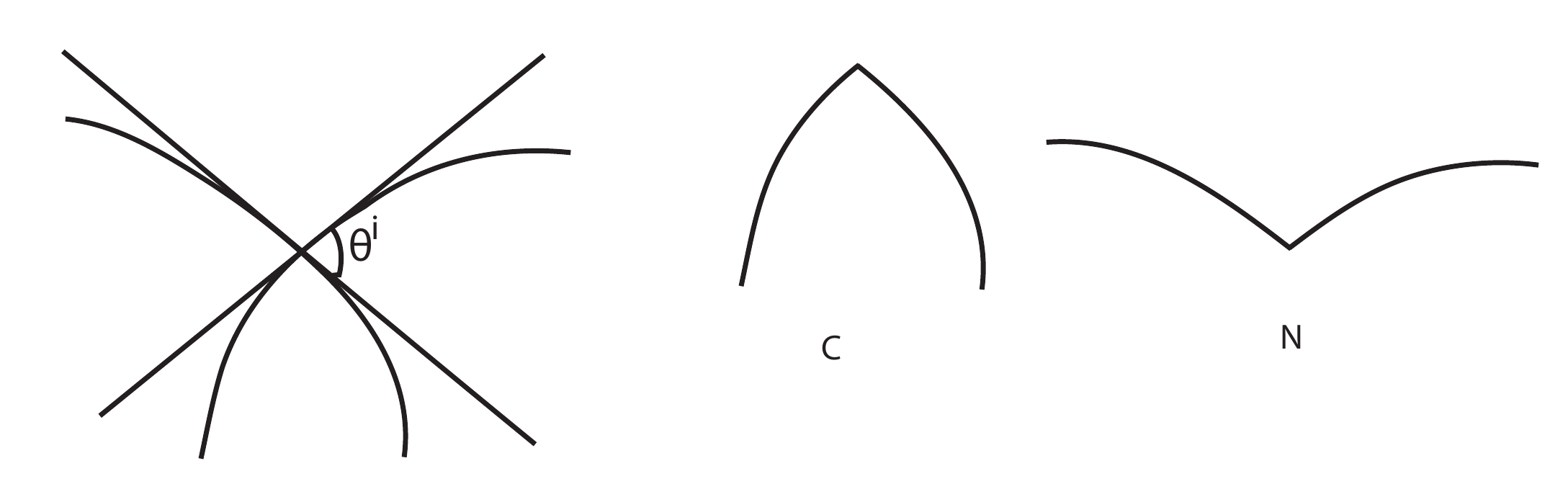}
\caption{Exterior Angles and Cusps}
\label{exteriorangle}
\end{figure}

When our Seifert circle is an o-circle, its total curvature is $2\pi$ plus or minus the sum of the exterior angles at each of its cusps. The exterior angles compensate for the angle of rotation between the tangent vectors at the cusp. Thus the total curvature of the o-circle is

\begin{eqnarray}
\kappa(O)&=&2\pi-\sum_{\text{C cusps}}\theta_C^i+\sum_{\text{N cusps}}\theta_N^j
\label{eqn:ocirc}
\end{eqnarray}

Now we consider an i-circle which has only three cusps. (Figure \ref{inniecurvature}). The exterior angles are labeled $\theta_1$, $\theta_2$, and $\theta_3$. The total curvature of any curve in $\mathbb{R}^2$ is the integral of $d\theta$, which (when we have no inflection points) equals the change in the angle between the normals to the smooth curve at each of its ends. Therefore, the curvature of each edge of the i-circle can be measured by calculating the angle between the two line segments which are orthogonal to that edge at the crossings that determine its end points. These angles are labeled $\phi_1$, $\phi_2$, and $\phi_3$ in Figure \ref{inniecurvature}.

\begin{figure}[here]
\includegraphics[scale=.7]{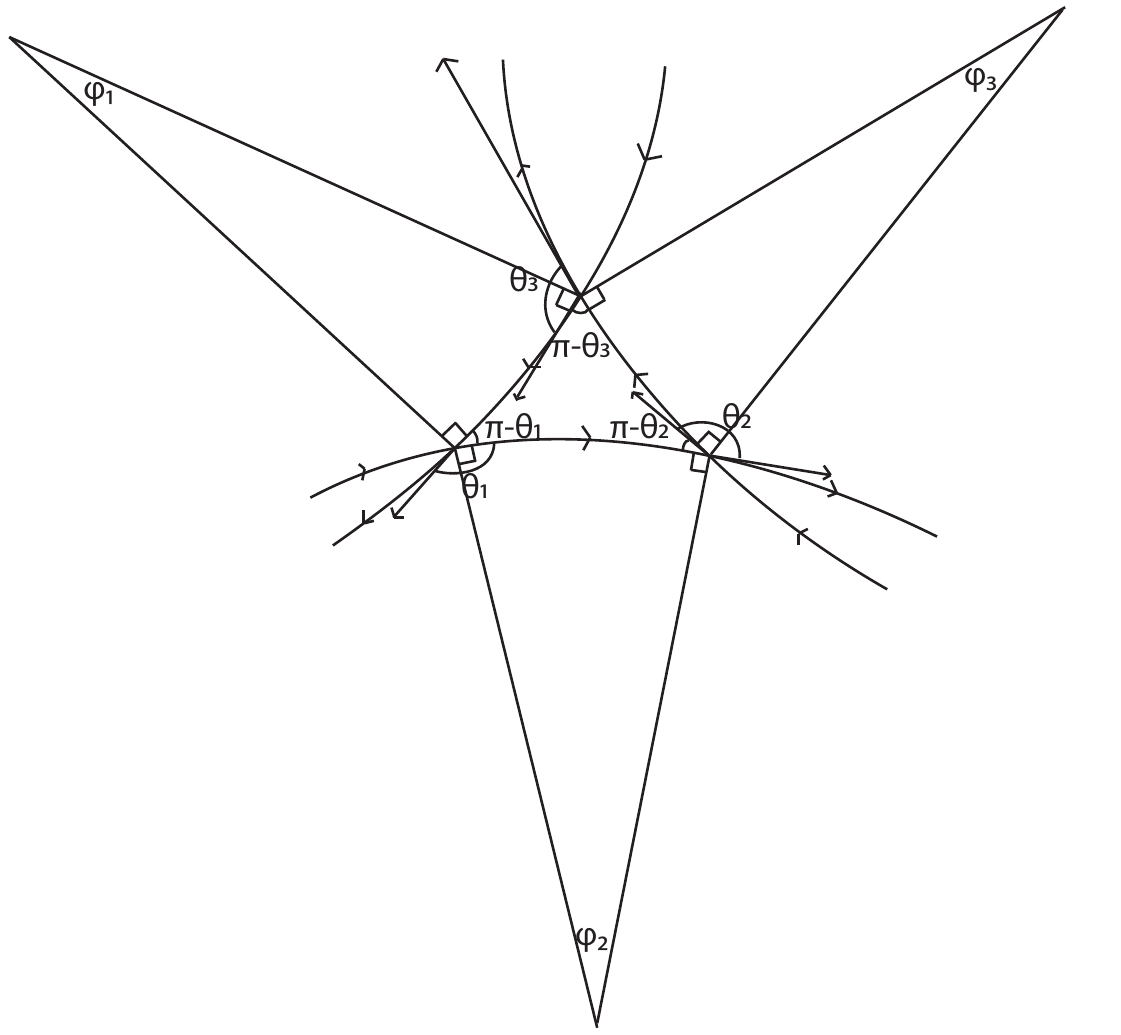}
\caption{Calculation of I-Circle Curvature}
\label{inniecurvature}
\end{figure}

Consider the hexagon in Figure \ref{inniecurvature}, whose interior angles sum to $4\pi$. Summing these angles we obtain the following equation:
$$4\pi =6\pi+\sum_i \phi_i -\sum_j \theta_N^j$$
Since we know $\sum \phi_i$ equals the curvature of the i-circle, this equals $-2\pi +\sum_i \theta_B^j$.  

Note that this formula can generalize to an i-circle with any number of sides as long as all the cusps are $N$ cusps. If we have such an i-circle with $n$ sides, we can put in orthogonal line segments as we did in the $3$ side case and create a $2n$-gon. Thus the sum of the angles is $(2n-2)\pi$. Here, there are $n$ obtuse angles which sum to $2n\pi-\sum_j \theta_N^j$ and $n$ acute angles which sum to $\sum_i \phi_i$. Therefore,
$$(2n-2)\pi=2n\pi-\sum_j\theta_N^j+\sum_i \phi_i$$
So the total curvature of any i-circle with only $N$ cusps is
$$\sum_i\phi_i=-2\pi+\sum_j \theta_N^j$$
Our final case is when i-circles have $C$ cusps, which occurs in the case of nested i-circles. This is pictured in Figure \ref{nestedfigure}.

\begin{figure}[here]
\includegraphics[scale=.7]{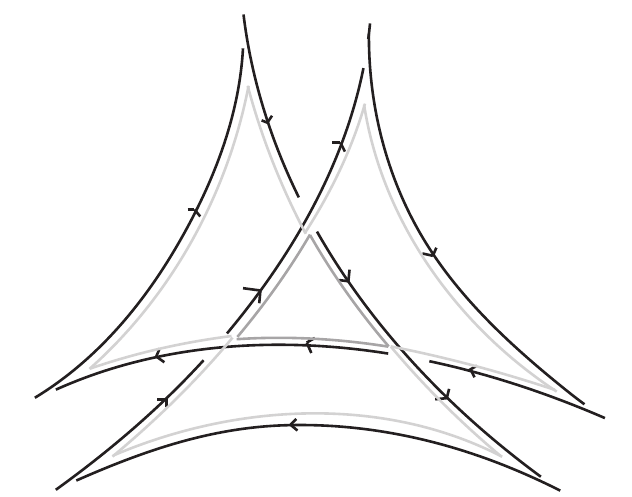}
\caption{Nested i-circles}
\label{nestedfigure}
\end{figure}

In this case, we will replace each $C$ cusp in the i-circle with a smooth curve that has the same tangent vectors at the end points, as in Figure \ref{Acuspinniefigure}.

\begin{figure}[here]
\includegraphics[scale= .7]{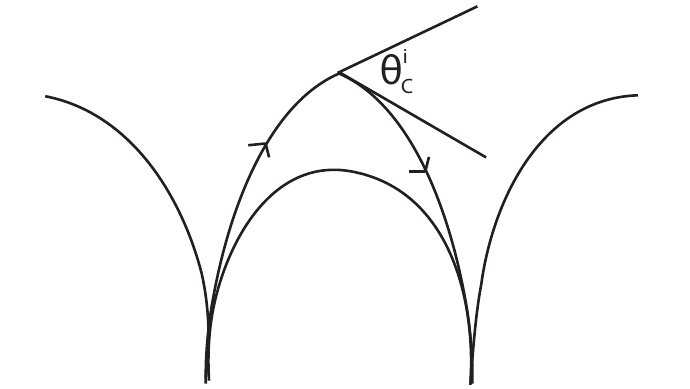}
\caption{Replacing a C cusp on an i-circle with a smooth curve}
\label{Acuspinniefigure}
\end{figure}

The curvature of each replacement curve is the curvature of the $C$ cusp that it replaced plus the exterior angle $\theta_C^i$. We repeat this process for each $C$ cusp. The smoothed out i-circle has only $N$ cusps, thus we can calculate its curvature from the above argument. We then subtract all of the $\theta_C$'s that we smoothed out to get the total curvature of the original i-circle

\begin{eqnarray}
\kappa(I)&=&-2\pi +\sum_{\text{N cusps}}\theta_N-\sum_{\text{C cusps}}\theta_C
\label{eqn:icirc}
\end{eqnarray}

The total curvature of the knot projection equals the sum of the curvatures of all the Seifert circles. Also, in both the formula for the curvature of an o-circle (Equation \ref{eqn:ocirc}) and in the formula for the curvature of an i-circle (Equation \ref{eqn:icirc}) all $\theta_C$'s are subtracted and all $\theta_N$'s are added. Therefore the total curvature of the knot $\kappa(P(K))$ is
$$\kappa(P(K))=2\pi o-2\pi i -\sum_i\theta_C^i+\sum_i\theta_N^i$$
Since each exterior angle is counted once as an $C$ cusp angle and once as a $N$ cusp angle, 
$$\sum_i\theta_C^i=\sum_i\theta_N^i$$
Therefore
$$\kappa(P(K))=2\pi (o-i)$$
By the previous lemma, $2\pi sp(P(K))=\kappa(K)$. So we obtain
$$sp(P(K))=o-i$$
\end{proof}

\begin{corollary} If a knot or link $K$ is curly, then any spiral projection of $K$ realizing $sp[K]$ must have at least one i-circle in its Seifert diagram.
\end{corollary}

\begin{proof}  By \cite{Yamada}, the minimum number of Seifert circles in any projection of $K$ equals $\beta [K]$. If we have a projection of $K$ realizing $sp[K]$ whose Seifert diagram has only o-circles, then by the previous theorem we have that $sp[K]$ is equal to the number of o-circles.  As there must be at least $\beta [K]$ o-circles, we have $sp[K]\geq \beta[K]$ and thus $sp[K]=\beta[K]$.
\end{proof}

We now use the above results to show that there are no $1$- or $2$- curly knots or links and to characterize all 3-curly knots and links.  First, if $sp[K]=1$, then $b[K]\leq 1$ and $K$ is the trivial knot, which has $\beta[K] = 1 = sp[K]$ and thus is not curly.  In order for a knot or link with $sp[K]=2$ to be curly, its minimal spiral projection must have a Seifert diagram with at least one i-circle.  Note that this i-circle must have more than 2 sides, as any bigon is an o-circle. Further, since an i-circle with four or more sides would force more than 2 maxima, the i-circle must be a triangle. If the triangle has a strand with two adjacent under/over-crossings, one can slide that strand over the third crossing of the triangle, converting the i-circle to an o-circle without changing the number of spirals, showing that the knot is not curly. Therefore, this triangle must be alternating. To complete this triangle to a nontrivial knot or link, we must add at least two additional local maxima. Thus the spiral index is greater than 2, contradicting the existence of 2-curly knots or links.  
 
\begin{theorem}
Every 3-curly knot or link has a minimal spiral projection of the form shown in Figure \ref{fig:3curlytheorem}.
\end{theorem}

\begin{figure}
\includegraphics[scale=.5]{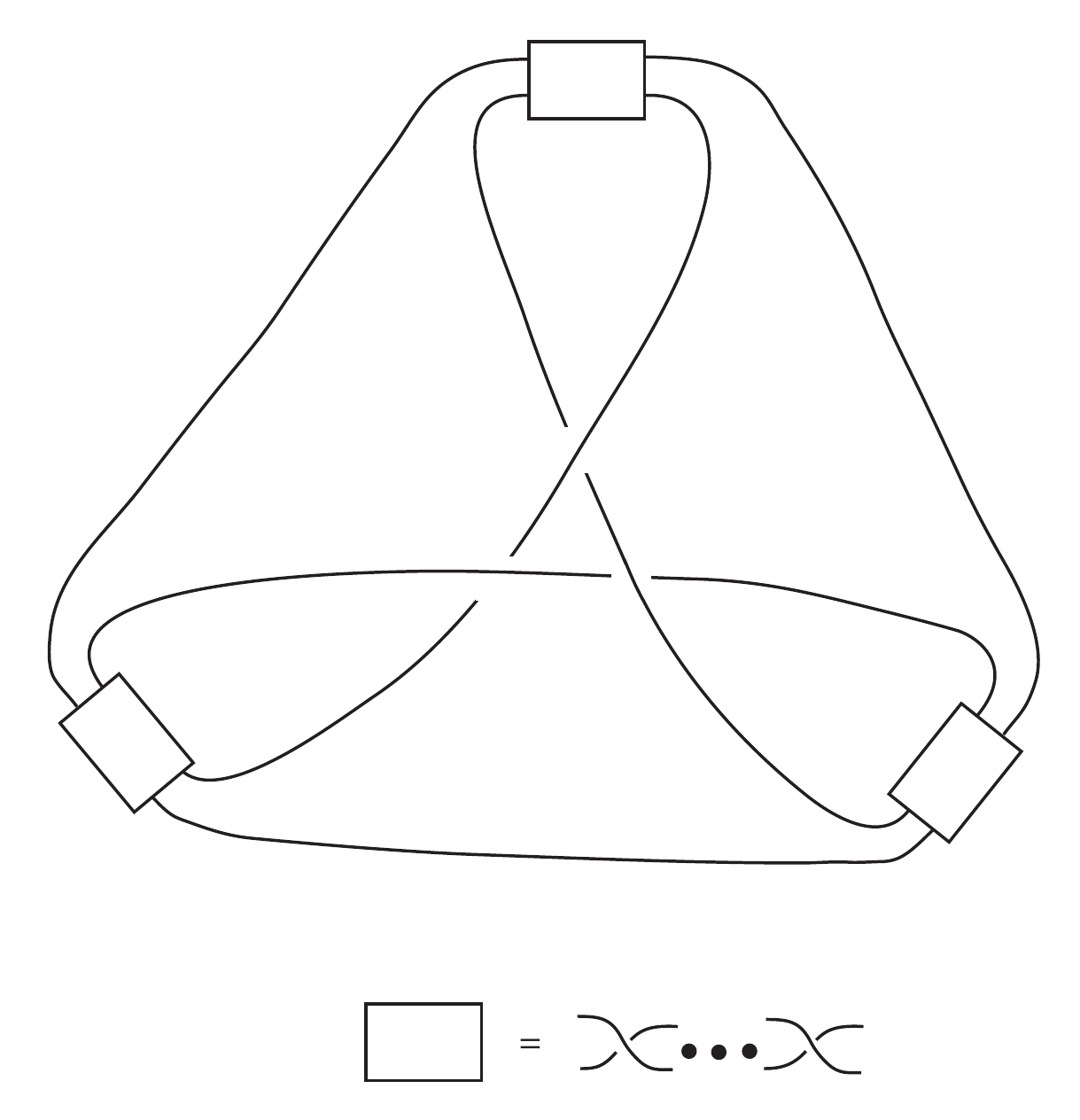}
\caption{Canonical 3-curly: Note to create a knot, two of the boxes contain twists with an even number of crossings and the third contains twists with an odd number of crossings.}
\label{fig:3curlytheorem}
\end{figure}

\begin{proof}
As with the 2-curly case, we require an alternating triangle in the link. (Given an i-circle with four or more sides, it is impossible to connect the sides to produce a projection with a spiral number of less than 4, so the i-circle must be a triangle).  Thus, all links with $sp[K]=3<\beta [K]$ must contain an alternating triangle as shown in Figure \ref{fig:3curlytheorem}. No other strands of the link may cross this part of the diagram.

We now consider all possible configurations that include such an alternating i-circle. In order to preserve constantly signed curvature, point $a$ can only connect to points $b$, $d$, or $f$.  Suppose first that $a$ connects to $b$, $c$ to $d$ and $e$ to $f$.  If $a$ links directly to $b$ without winding around the triangle at the center, and the other pairs of points do likewise, we have a trivial knot.  Thus, we can assume $a$ loops once around the triangle and then connect to $b$ as shown.  If $c$ or $e$ does likewise, or if the strand from $a$ to $b$ loops around twice, the projection will end up with a spiral number number greater than 3, and so $c$ must link straight to $d$ and $e$ to $f$.  This gives us a  3-curly projection of a knot that is included in the possibilities given by Figure \ref{fig:3curlytheorem}. 

If $a$ connects to $b$, the only other possibility is that $c$ connects to $f$ and $e$ to $d$. This generates a 2-component link. Since the strand connecting $e$ to $d$ wraps around the triangle, neither of the other two connecting strands can do so, and we obtain a projection that is included in the possibilities in Figure \ref{fig:3curlytheorem}. Up to symmetry of the central triangle, this case covers all the possibilities for a 2-component link.

The only other possibility is if $a$ connects to $d$, $b$ to $e$ and $c$ to $f$. In this case, none of the connecting strands can wrap around the triangle or the spiral number will be greater than 3. This results in a 3-component link and it also is included in the possibilities provided by Figure \ref{fig:3curlytheorem}.

\begin{figure}
\includegraphics[scale=.15]{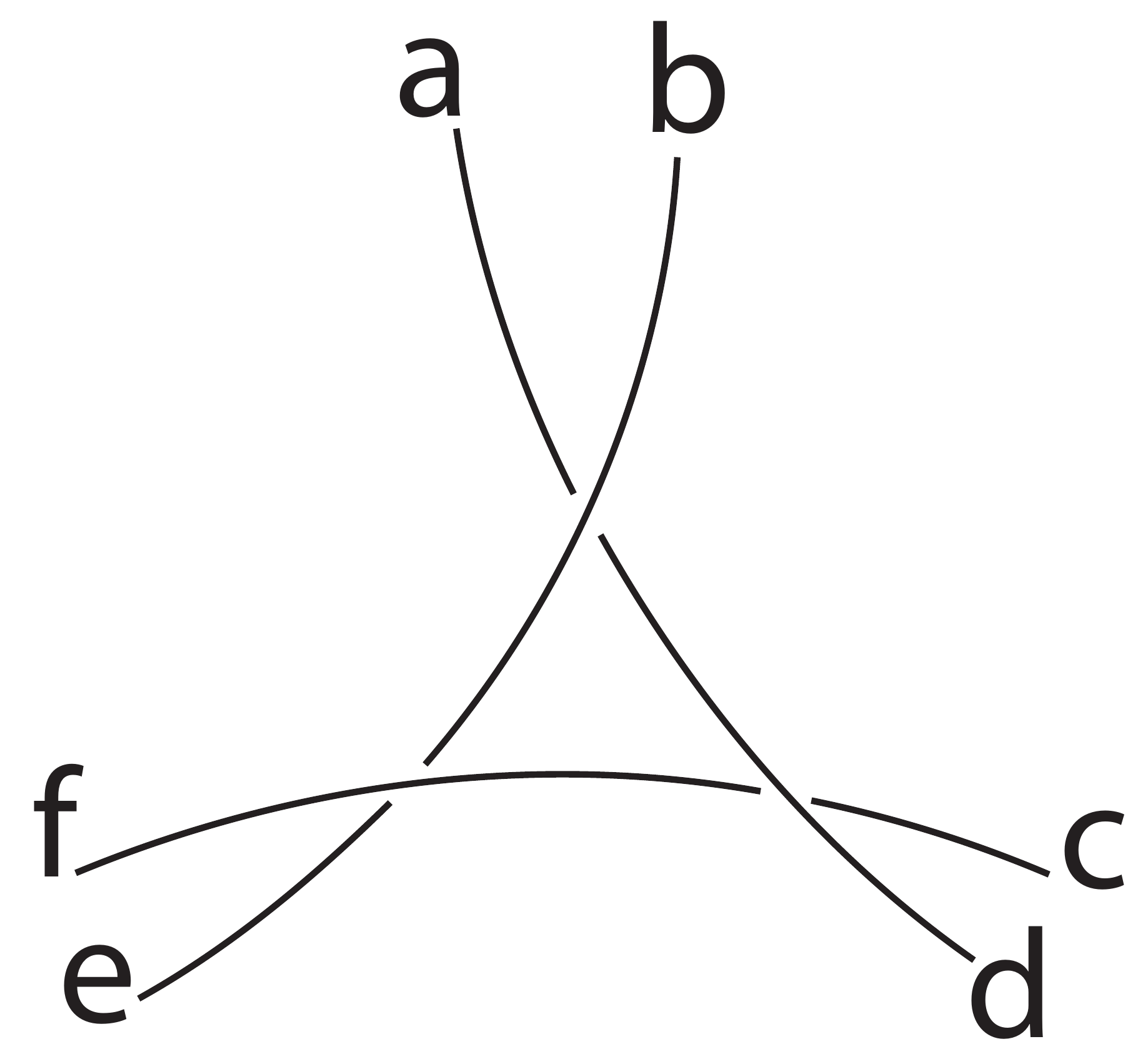}
\caption{i-circle}
\label{3curlyproof}
\end{figure}

Note all crossings outside of the middle triangle can be arbitrarily chosen.  Furthermore, there are three sections of the knot where strands with the same orientation lie alongside one another, and so arbitrarily many positive or negative twists can be added between them.  In the same way we can convert a braid form projection to a spiral projection, we can isotope the knot to prevent these twists from adding inflection points.  Thus, the most general form of a 3-curly knot or link is as in Figure \ref{fig:3curlytheorem}, where the boxes represent any number of positive or negative two-twists.
\end{proof}

\begin{theorem}  All 3-curly links are alternating and prime.
\label{altprime}
\end{theorem}

\begin{figure}
\includegraphics[scale=.5]{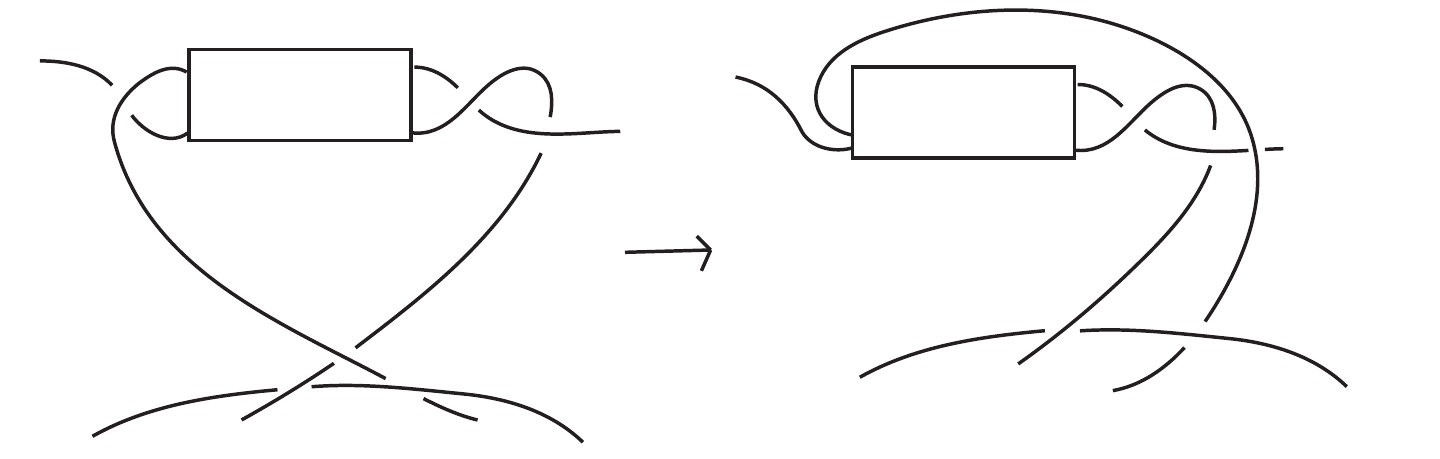}
\caption{Putting 3-curly knots in alternating form.}
\label{alternate1}
\end{figure}

\begin{figure}
\includegraphics[scale=.4]{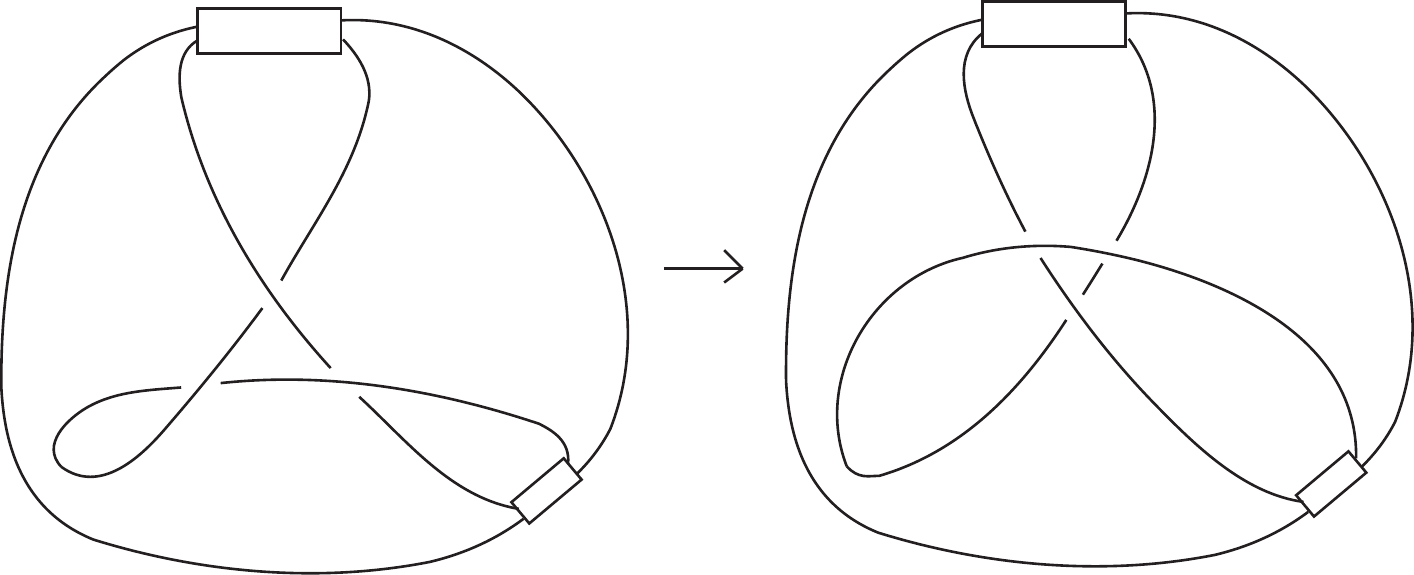}
\caption{Changing to a 3-braid form.}
\label{alternate2}
\end{figure}


\begin{proof}  We may assume that the link appears as in Figure \ref{fig:3curlytheorem} and that the crossings within a given box alternate.  If the crossings in a given box are alternating with respect to the adjacent crossings of the i-circle, then we need not alter that portion of the projection.  Suppose now that the crossings in the top box are not alternating with respect to the i-circle crossings. If there is only one crossing in the top box, then the knot is the composition of two $2$-braid links, and must have $\beta = 3$.  Hence it is not curly. 

Assume now there are at least two crossings in the top box.  By performing the move shown in Figure \ref{alternate1}, we cause that region of the link to become alternating relative to the triangle.

We now consider each of the lower boxes, and assume that they are not already alternating with respect to the center of the knot.  If there are no alternating crossings in the box then the knot can be put in $3$-braid form, as shown in Figure \ref{alternate2} and thus is not curly.  Assuming there are crossings, we can perform a similar move to that shown in Figure \ref{alternate1}, causing that region to be alternating relative to the triangle.  We can thus obtain an alternating projection.  Note that by \cite{Alternate}, the link must be in minimal crossing form.  Further, as we are in minimal crossing alternating form and the projection of the link is not obviously composite, the link is prime \cite{Altprime}.
\end{proof}

In addition to being the minimal spiral form of any $3$-curly knot, Figure \ref{fig:3curlytheorem} allows us to generate all $3$-curly knots up to a given crossing number.  Up to adding arbitrarily many twists, there are 8 possible configurations of this knot (given that there are three locations where crossing can be changed between over and under).  It can be shown that each of these configurations can be converted into an alternating knot whose number of crossings increases as crossings are added to any of the boxes.  Thus, it is possible to go through every possible combination of varying numbers of positive or negative twists to generate all $3$-curly knots up to a given crossing number.  We have determined all $3$-curly knots up to nine crossings, as listed in the introduction. These knots, along with the $3$-braid knots, constitute all knots of spiral index $3$.

\begin{proposition} All $3$-curly knots and links have braid index $4$.
\end{proposition}

\begin{proof}  Consider a link of the form shown in Figure \ref{fig:3curlytheorem}.  By performing the move shown in Figure \ref{twirl} on the alternating triangle, we can put the link in a $4$-braid form, meaning that the braid of any link of this form has a braid index of at most $4$.  As $3$-curly links must be of this form and by definition have braid index at least $4$, we conclude that all 3-curly links have braid index $4$.
\end{proof}

\begin{figure}
\includegraphics[scale=.5]{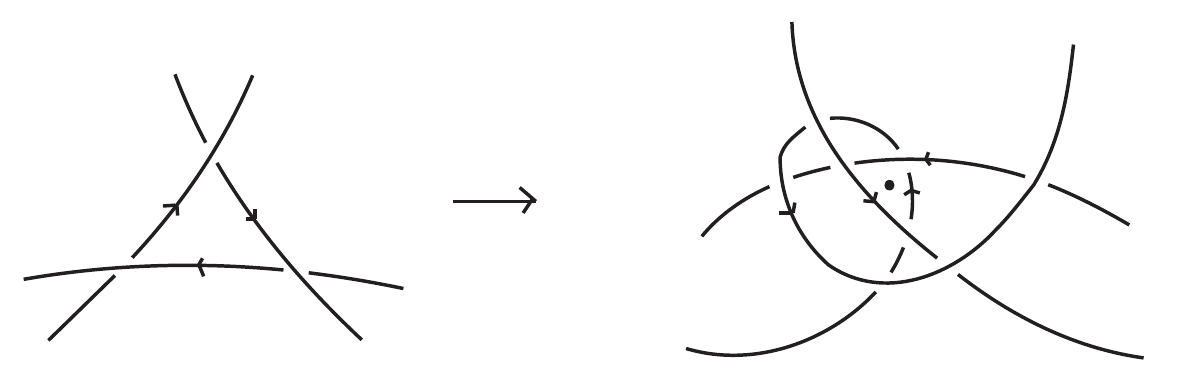}
\caption{Changing a $3$-curly knot from $3$-spiral form to $4$-braid form.}
\label{twirl}
\end{figure}

Although a 3-spiral knot can have at most a difference of $1$ between $\beta [K]$ and $sp [K]$, this does not hold in general.  For instance, Figure \ref{nestedb} shows a curly knot with a nested pair of i-circles, spiral index 4 and braid index 8.  Although a general bound on braid index in terms of spiral index has not yet been found, we conjecture that a quadratic bound exists. 

In Figure \ref{spiralbraid}, we see a representative of a class of knots, which should yield braid index growing quadratically with spiral index. The knots should be taken to be alternating. The knot shown has spiral number 8 in this projection. It's  21 i-circles are shaded. Let $K_n$ be the corresponding generalization to an alternating knot in this class with spiral number at most $n$, where $n \ge 3$. Note that the number of i-circles for such a projection is $\frac{(n-1)(n-2)}{2}$. Utilizing the move from Figure \ref{spiralbraid} on each i-circle and then sliding strands to put the knot in braid form we can see that $\beta[K_n] \leq n + i \leq \frac{n^2 -n +2)}{2} $.

In fact, we expect that this is equality, and utilizing the MFW inequality (c.f.\cite{MFW},\cite{MFW2}) to bound braid index from below, one can see this to be true for $n=3,4,5$. It remains to prove the lower bound on braid index for $n \geq 6$. 

\begin{figure}
\includegraphics[scale=.3]{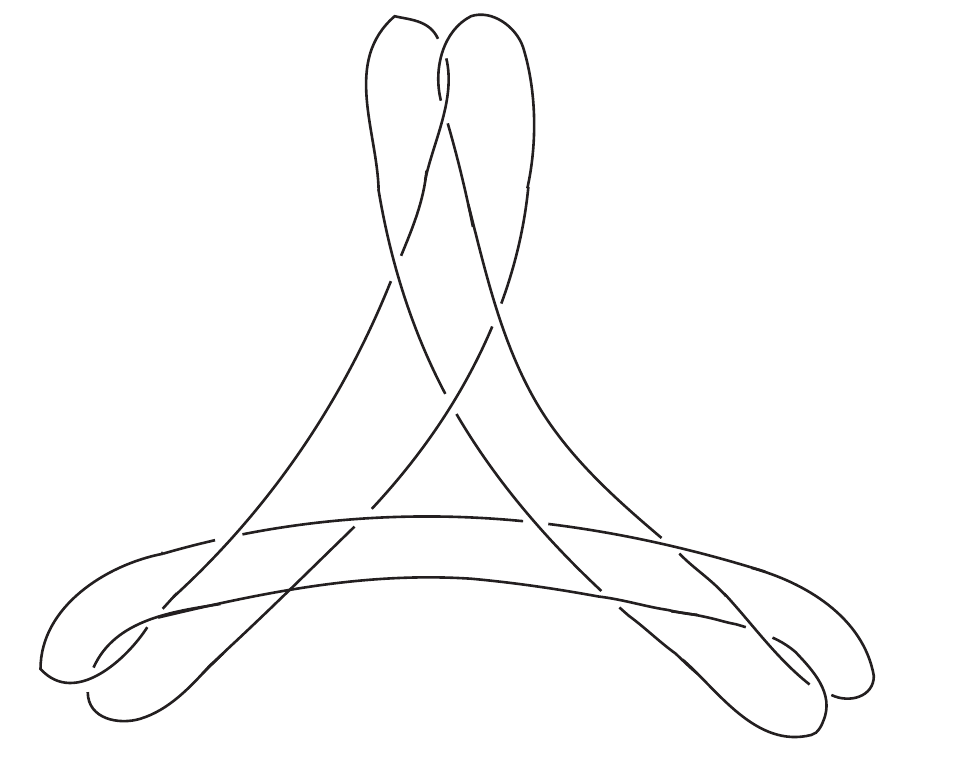}
\caption{A Knot with $sp[K]=4$ and $\beta[K]=8$.}
\label{nestedb}
\end{figure}

\begin{figure}
\includegraphics[scale=.4]{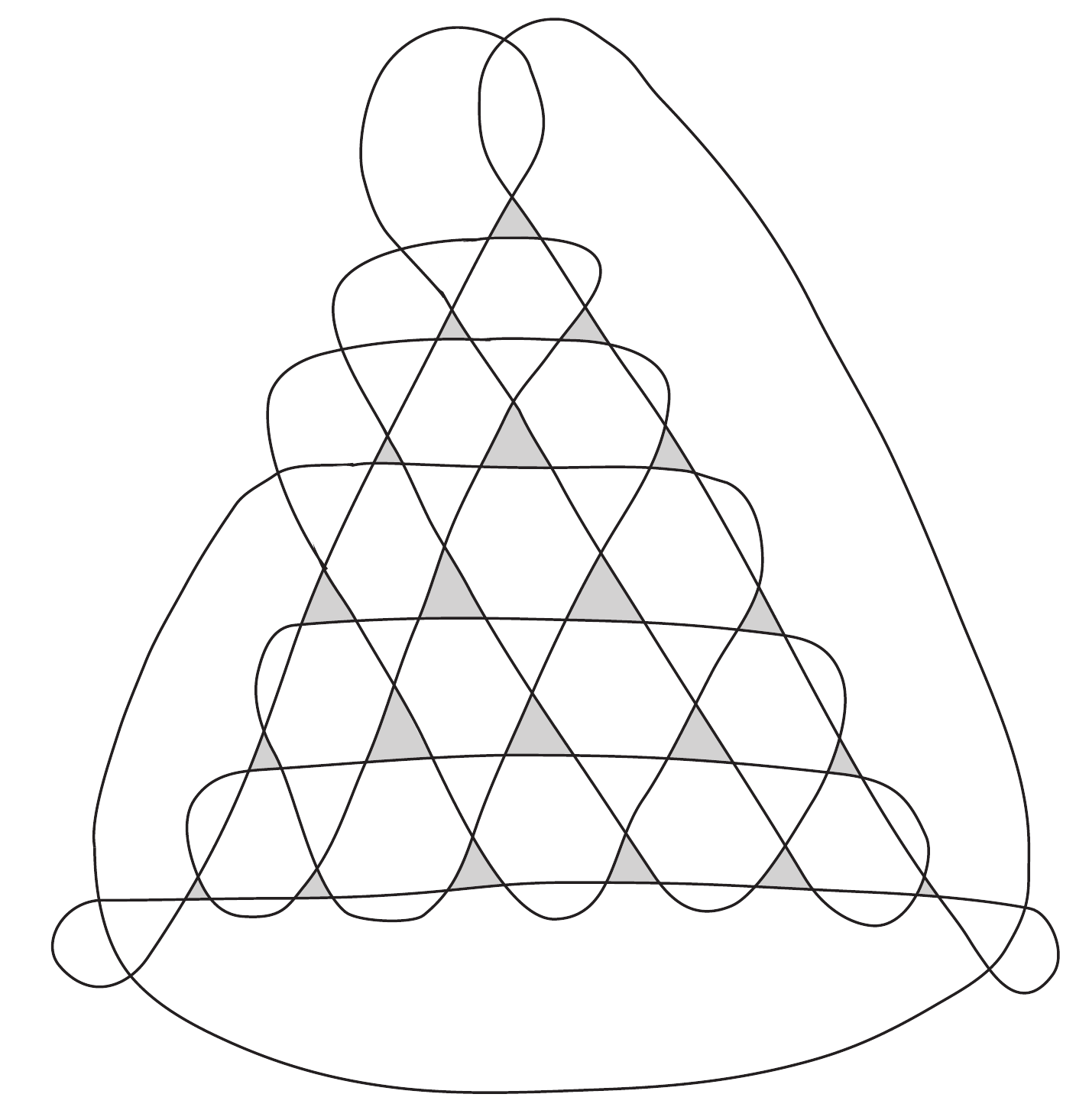}
\caption{The knot $K_8$ in the family $\{K_n\}$.}
\label{spiralbraid}
\end{figure}

The following proposition provides further evidence demonstrating the abundance of curly knots.

\begin{proposition} All twist knots of at least six crossings are curly.
\end{proposition}

\begin{proof}  We denote a twist knot by $T_c$, where $c$ is the number of crossings in the ``twist'' portion of the knot (equivalently, $T_c$ is the $(c+2)$-crossing twist knot.) See Figure \ref{twistknots}

We will first prove the result for odd twist knots. The HOMFLYPT polynomial of $T_n$ for odd n is
$$P_{T_n} (v,z)=v^{2}+v^{n+1}-v^{n+3}-z^2(v^{n+1}+v^{n-1}+...+1)$$
Recall that the MFW inequality gives us that 
$\beta[K] \geq \frac{1}{2}\text{v-breadth}P_K(v,z)+1$.
This gives us the bound $\beta [T_n] \geq \frac{n+3}{2}$.  As shown in Figure \ref{twistknots}, we can put any odd-crossing twist knot with at least $7$ crossings in a spiral form with $sp(T_n)= \frac {n+1}{2}$, giving us the bound $sp[T_n]\leq \frac {n+1}{2}$. Thus $sp[T_n]<\beta [T_n]$ for odd $n$ and $T_n$ is curly.  
The HOMFLYPT polynomial of $T_m$ where $m$ is even is
$$P_{T_m} (v,z)=v^{-2}+v^m-v^{m-2}-z^2(v^{m-2}+v^{m-4}+...+1)$$
The MFW inequality  gives us that 
$\beta[K] \geq \frac{1}{2}\text{v-breadth}P_K(v,z)+1$.
Thus, we find that 
$$\beta [T_m] \geq \frac{m+4}{2}$$
We can similarly put any even-crossing twist knot with at least $6$ crossings in a spiral form with $sp(T_m)=\frac {m+2}{2}$.  Therefore $sp[T_m]\leq \frac{m+2}{2}$.  Thus we have $sp[T_m]\leq \frac{m+2}{2}<\frac{m+4}{2}\leq \beta [T_m]$, so $sp[T_m]<\beta [T_m]$ for even m and $T_m$ is curly. We conclude that all twist knots with at least 6 crossings are curly. Note that none of $3_1$, $4_1$ or $5_2$ can be curly since they all have braid index $2$ or $3$.
\end{proof}

\begin{figure}
\includegraphics[scale=.5]{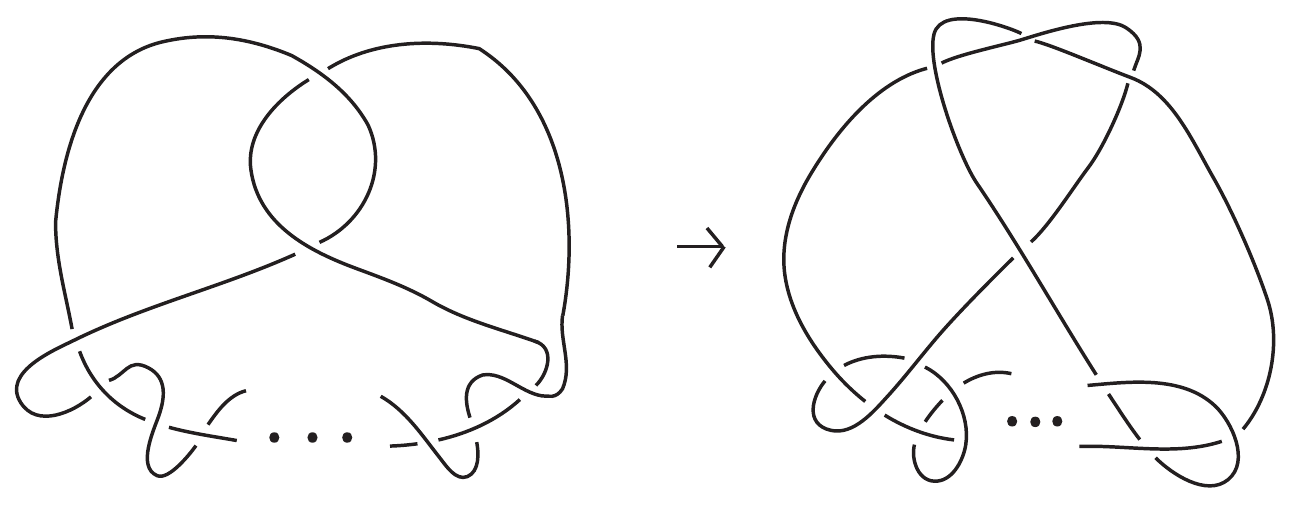}
\caption{Putting Odd Twists Knots in Spiral Form}
\label{twistknots}
\end{figure}

\begin{proposition}  $sp[K_1\#K_2] \leq sp[K_1]+sp[K_2]-1$
\end{proposition}

\begin{proof}Take projections of conformations of $K_1$ and $K_2$ in minimal spiral form, $\alpha_1(t)=P_1(K_1(t))$ and $\alpha_2(t)=P_2(K_2(t))$.

We define an exterior arc to be an arc of the projection with no crossings except possibly at its end points such that the line that passes through its endpoints divides the plane so that the arc is the only piece of the knot on that side of the line.

We first claim that any knot projection has an exterior arc. Given a projection lying in $\mathbb{R}^2$, we can find such an arc by taking a horizontal line above the knot and lowering it until it touches the knot tangentially. Here it will intersect the knot at finitely many points. Take the point whose $x$ coordinate is the greatest; say this is $\alpha(t_0)$. Without loss of generality say we have chosen a parameterization with orientation going in the positive $x$ direction at $\alpha(t_0)$. Then there exists an $\varepsilon >0$ such that  the set of tangent lines to points in $\alpha((t_0,t_0+\varepsilon))$ all intersect the knot in exactly one place (the point of tangency). A subset of those tangency points will be an exterior arc.

If $\alpha_1$ or $\alpha_2$ has an exterior arc that subtends an angle of $\geq \pi$, we do no transformation to that knot projection. 
If a projection has no such exterior arc, we take an exterior arc that subtends a smaller angle. Since the arc subtends an angle less than $\pi$, the tangent lines to the curve at the endpoints of the arc intersect on the exterior side of the plane. Put the plane of projection into $\mathbb{R}^3$ such that the line through the end-points is along the $x$-axis and the line perpendicular to the $x$-axis that goes through the intersection point of the tangent lines be the $z$-axis. Let $(0,0,h)$ be the intersection point of the tangent lines. Given this set-up, we perform a linear transformation onto the $xy$-plane given by:
$$T(x,0,z)=(\frac{xh}{h-z}, \frac{z}{h-z}, 0)$$
This transformation sends each point $p$ in the $xz$-plane where $z<h$ to the point where the ray from $(0,-1,h)$ and through $p$ intersects the $xy$-plane. Notice that this sends the tangent lines to the arc at its endpoints to parallel lines therefore the projection of the arc subtends an angle of $\pi$.

Now we claim that $T$ preserves minimal spiral form. First of all since $\alpha_i)$ has no $s$-inflection points, $T(\alpha_i))$ has no $s$-inflection points because if $T(\alpha_i))$ had some $s$-inflection point, $p_0=T(\alpha_i(t_0)))$, then for any $\varepsilon$ neighborhood of $t_0$, $B_{\varepsilon}$, the tangent line to the curve at $p_0$ would have pieces of $T(\alpha_i(B_{\varepsilon})))$ to either side of the line. The plane that passes through the tangent line at $p_0$ and the tangent line at $T^{-1}(p_0)$ divides $\mathbb{R}^3$ into two halves. T preserves the division into these two halves. However since $\alpha_i)$ has no $s$-inflection points, there exists a $\delta$ neighborhood of $t_0$, $B_{\delta}$ such that $\alpha_i(B_{\delta}))$ only on one side of this plane. Therefore $T(\alpha_i(B_{\delta})))$ is only on one side of the plane, and thus only on one side of the tangent line so $p_0$ is not an $s$-inflection point.

Note that if we consider any map in the family of linear transformations
$$\{f_s(x,0,z)=(1-s)P(x,0,z)+s(x,0,z)=(x+\frac{sxz}{h-z},\frac{sz}{h-z},z-tz)\}$$
and apply any of these $f_s$ (where $s\in [0,1]$) to the spiral knot projection, there will be no inflection points in the image.

Finally, we assert that the transformation maintains \emph{minimal} spiral form. Assume to the contrary, that the spiral number of the resulting projection is greater than the spiral number of the original projection. Then there must exist some $s_0$ such that the image of the knot projection under $f_{s_0-\varepsilon}$ differs in spiral number by one from the image of the knot projection under $f_{s_0+\varepsilon}$ for all $\varepsilon>0$ which are sufficiently small. The only way to transition to a higher spiral number requires a cusp. However, since $\alpha_i$ is smooth, $f_s\circ \alpha_i$ is also smooth for any $s\in [0,1]$ and therefore has no cusps. 

Thus $T$ preserves minimal spiral form and gives us a projection of the knot which has an exterior arc which subtends an angle of at least $\pi$. Using these spiral projections of $K_1$ and $K_2$ we can compose them by deleting the exterior arc subtending an angle of $\pi$ and connecting the loose ends with straight horizontal lines. This yields a spiral form of $K_1\# K_2$ which has spiral number equal to $sp(K_1)+sp(K_2)-1$ giving us the upper bound.
\end{proof}

\begin{corollary}  There exist knots with arbitrarily large gaps between $sp[K]$ and $\beta [K]$.
\end{corollary}

\begin{proof} Consider the knot $6_1$, which has spiral index 3 and braid index 4, and compose it with $n-1$ copies of itself. Denote the resulting knot $n6_1$. Since braid index is subadditive under composition, $\beta[n6_1]$ increases by $3$ as $n$ increases by $1$. Since the spiral form of $6_1$ has two exterior arcs of curvature $\pi$, we can compose it with itself $n$ times along these arcs and obtain a spiral form whose spiral number is $nsp[6_1]-n$ which means the spiral index can increase by at most $2$ for each $n$. Therefore the braid index becomes arbitrarily higher than the spiral index as $n$ increases. 
\end{proof}

Note that an identical proof holds for any combination of curly knots composed with one another.

\section{Projective Superbridge Number}

We begin this section by introducing the projective superbridge number of a knot. This invariant minimizes the greatest number of maxima of a knot seen in a given projection.

\begin{definition}
Let $K$ be a conformation of a knot in 3-space and, given $v\in S^2$, let $P_{v}(K)$ be the  projection of $K$ onto a plane $P$ orthogonal $v$. Now let $\mu(w,P_{v}(K))$ denote the number of maxima of the knot projection with respect to its projection onto the vector $w\in S^1\subset P$.
$$ psb[K] = \min_{K\in[K]} \min_{v\in S^2} \max_{w\in S^1}\mu(w,P_vK)$$
\end{definition}

Note that we can define bridge index analogously as
$$b[K]=\min_{K \in [K]} \min_{v \in S^2} \min_{w \in S^1} \mu(w,P_vK)$$

Moreover, the superbridge index of \cite{Kuiper} is given by 

$$sb[K]=\min_{K \in [K]} \max_{v \in S^2} \max_{w \in S^1} \mu(w,P_vK)$$

It follows immediately that $b[K] \leq psb[K] \leq sb[K] $.
 
One of our primary interests in projective superbridge number is its relation to the stick index, $s[K]$, of a knot $K$. Stick index is the least number of line segments in 3-space, attached end to end, needed to realize a conformation of the given knot. The following result is useful in determining lower bounds for stick index.

\begin{proposition}
For a knot $K$, we have
$$s[K] \geq 2psb[K]+1$$
\end{proposition}

\begin{proof}
Let $K$ be in minimal stick conformation and let $v\in S^2$ be a vector parallel to some edge of $K$. Then in the projection $P_{v}(K)$ only $s[K]-1$ sticks are visible since we have projected directly down the chosen stick. Because in a generic stick projection extrema will only occur at the vertices, and half the extrema must be minima, our projection realizes at most $\frac{s[K]-1}{2}$ maxima relative to any vector in the plane.  Further, there must be some vector $w\in S^1$ relative to which we see at least $psb[K]$ maxima.  After minimizing over all conformations, we obtain the desired inequality.
\end {proof}

We now turn to the relationship between projective superbridge number and spiral index. Since a projection of a knot in spiral form has the same number of maxima over all directions in the plane, the greatest number of maxima in that projection is equal to the spiral number of that projection. After we minimize over all projections and all conformations, we see $psb[K] \leq sp[K]$. Because projective superbridge serves as a lower bound on stick number, we are interested in knots, $K$, such that $b[K] < psb[K]$. For this case we see that $s[K] \geq 2(b[K]+1) +1 = 2b[K] + 3$. To this end, we have the following proposition.

\begin{proposition}
\label{prop:b=psbimpliessp=b}
If $K$ is a knot such that $b[K] = psb[K]$, then $sp[K]=b[K]$.
\end{proposition}

\begin{proof}
Let $b[K] = psb[K]$. Then there is a projection $P(K)$ of $K$ that realizes both $b[K]$ and $psb[K]$. Thus 
$$\min_{w \in S^1} \mu(w,P(K)) = \max_{w \in S^1} \mu(w,P(K))$$
This implies  $\mu(w,C)$ is constant over all $w \in S^1\subset P$. Hence, $P(K)$ is in spiral form with $\mu(w,P(K))= b[K]$ spirals so $sp[K] \le b[K]$. Since the reverse inequality always holds, the proposition follows.
\end{proof}

An example will help to illustrate the significance of the proposition above.

\begin{example}
\end{example}
Consider $4_1 \# n3_1$, the composition of the figure eight knot with n copies of the trefoil. This knot can be realized with $7+2n$ sticks. 
Since both bridge index and braid index are subadditive under composition and $b[3_1]=b[4_1]=\beta[3_1]=2$, $\beta[4_1]=3$, we have
$$b[4_1 \# n3_1]=n+2 $$
$$\beta[4_1 \# n3_1]=n+3.$$
If one could show  $b[4_1 \# n3_1] < psb[4_1 \# n3_1]$  or $b[4_1 \# n3_1] < sp[4_1 \# n3_1]$ it would follow that $s[4_1 \# n3_1]\ge 7+2n$. This would provide the necessary lower bound in order to obtain the stick indices for this infinite family of knots. Note that we know that this is the case when $n=1$.\\

Thus proving a subaddivity rule for spiral number under composition akin to that of braid index or bridge index would lead not only to a better understanding of that invariant but also to additional tools for calculating stick number.

Although a similar subadditivity result for projective superbridge index would give us useful bounds, we can in fact show by counterexample that projective superbridge index is not subadditive  under composition.

\begin{proposition}
The projective superbridge number of a knot is not always additive minus one under composition
\end{proposition}

\begin{proof}
Consider the knot $K=9_{46}$, which has the following properties:
\begin{eqnarray*}
b[K]&=&3\\
psb[K]&=& 4\\
s[K]&=&9
\end{eqnarray*}
We know $sp[K]=4$ ($K$ is not a $3$-spiral knot and it is possible to put it in $4$-spiral form). Apply \ref{prop:b=psbimpliessp=b} to get $psb[K]=4$.  
Assume for sake of contradiction that $psb[K_1 \# K_2]=psb[K_1]+psb[K_2]-1$.
Then $$psb[9_{46}\# 9_{46}]=4+4-1=7$$
Since $2psb[K]+1\leq s[K]$,
we have $$2(7)+1 =15 \leq s[9_{46}]$$
On the other hand $9_{46}$ satisfies the conditions described in \cite{Exteriorsticks} so that we can save four sticks in composition to obtain the following upper bound on stick number:
$$s[K]\leq 9+9-4=14$$
This implies that $15 \leq s[K] \leq 14$, providing a contradiction.
\end{proof}

\section{Milnor's Curvature-Torsion Invariant}
We now relate the spiral index and the projective superbridge number of a knot to the curvature-torsion invariant defined by Milnor in \cite{Milnorcurves}.  
For a knot $K$, denote the number of inflection points in the projection $P_{v}(K)$ by $\nu(P_{v}(K))$, and the number of \emph{extrema} in that projection with respect to a vector $w$ by $\hat{\mu}(w,P_{v}(K))$. Notice that for any projection $P_{v}(K)$ and for any $w$, $\hat{\mu}(w,P_{v}(K)) = 2\mu(w,P_{v}(K))$.

Milnor defined the invariant \emph{curvature-torsion} of a knot $K$ as
$$(\kappa+\tau)[K] = \inf_{K \in [K]}\int_{K}(\kappa+ |\tau|)$$
More recently, Honma and Saeki proved the following theorem, which characterizes $(\kappa + \tau)[K]$ in terms of the number of critical points of the knot projection.
\begin{theorem}
\label{characterizekappatau}
\cite{Honma} Let $K$ be a knot. Then  
\[\frac{1}{\pi}(\kappa+\tau)[K]= \min_{K \in [K]} \min_{v \in S^2} \min_ {w \in S^1} (\hat{\mu}(w,P_{v}(K)) + \nu(P_{v}(K)))\]
\end{theorem}

This characterization of $(\kappa + \tau)[K]$ allows us to relate this invariant to projective superbridge and allows us to answer an open question about $(\kappa + \tau)[K]$ that was posed by Honma and Saeki.

\begin{theorem} 
\label{psbltecurvtors}
For a knot $K$, we have 
\[2\pi psb[K]\leq(\kappa + \tau)[K] \]
\end{theorem}

Let $P(K)$ be a smooth oriented projection of a knot conformation $K$ into some plane $P$. Then extrema in this projection are measured relative to vectors $w\in S^1\subset P$. We call $v\in S^1$ a \emph{local extrema changing vector} at $x\in P(K)$ if for any neighborhood $x\in B\subset P(K)$, 
\begin{align}\label{extremachanging}\hat{\mu}(v+\varepsilon, B)= \hat{\mu}(v-\varepsilon, B)\pm 2n \end{align}
for all sufficiently small $\varepsilon$. We will call $n$ the local multiplicity of $v$ at $x$. We call $v$ an extrema changing vector with multiplicity $n$ if (\ref{extremachanging}) holds for $B=P(K).$

Denote the unit tangent vector to $P(K)$ at $p$ by $T_p$. Define $N_p$ to be the rotation of $T_p$ by $\frac{\pi}{2}$ in the counterclockwise direction. Given any point $p \in P(K)$, $p$ is a \emph{critical point} with respect to a vector $w$ when $T_p$ is orthogonal to $w$. Assuming that $T_p$ is orthogonal to $w$, note that if $T_q$, for a point $q\neq p$, is some oriented angle $\theta$ away from $T_p$ then $q$ is a critical point with respect to some $\tilde{w}$ which is $\theta$ away from $w$. This corresponds to rotating the Frenet frame by the angle $\theta$. Further, if our curve has nonzero curvature at $q$ then $q$ is an extremum with respect to $\tilde{w}$. To prove our theorem we will need the following lemma.

\begin{lemma}
Let $\alpha$ be a immersed curve in $\mathbb{R}^2$ with no $u$-inflection points. Then for each extrema changing vector $v_0$ with multiplicity $n$ there are $n$ $s$-inflection points $x_i$ such that $v_0= \pm N_{x_i}$. Further, if $x$ is an $s$-inflection point on $\alpha$, then $\pm N_{x}$ are local extrema changing vectors.
\end{lemma}

\begin{proof}
Let $\alpha(t)$ be an immersed curve with no $u$-inflection points. Given some extrema changing vector, $v_0\in S^1$, consider the set of points, $\alpha(t_1)=x_1,...,\alpha(t_n)=x_n$ such that $v_0=\pm N_{x_i}$. These are the critical points on $\alpha$ with respect to $v_0$. Let $\varepsilon>0$ and let $B(v_0,\varepsilon)$ be a neighborhood of $v_0$ in $S^1$ of all vectors whose angle from $v_0$ is less than $\varepsilon$. We first show that for any $\alpha(t_0)=y\neq x_i$, there is a neighborhood $B_y\subset \alpha$ of $y$ such that for any $z\in B_y$, $z$ is not a critical point with respect to any $v\in B(v_0,\varepsilon)$.

Consider the continuous function
$$g:S^1\times S^1 \rightarrow \mathbb{R}$$
defined by $g(t,v)=\alpha'(t)\cdot v$.
Let $y=\alpha(t_0)\neq x_i$ for any $i$. Then $g(t_0,v_0)\neq 0$ and there exists an open product neighborhood $U\times V$ of $(t_0,v_0)$ such that $g(s,v)\neq 0$ for any $(s,v)\in U\times V$. Set $B_y=\alpha(U)$. Therefore there is a neighborhood of $\alpha(t_0)$ such that none of the points in that neighborhood are critical points with respect to any $v$ in some neighborhood of $v_0$. Thus, in order to locate the critical points of $\alpha$ with respect to vectors close to $v_0$ it suffices to restrict ourselves to neighborhoods of the points $x_i$.

Now we prove that the only possible extrema changing points with respect to $v_0$ are places where curvature is zero. Reorder so that $\alpha(t_1)=x_1,...,\alpha(t_k)=x_k$ are all the critical points which are not inflection points. Using the same function $g$ as above, $g(t_i,v_0)=0$  and 
$$\frac{\partial g}{\partial t}(t_i,v_0)=\alpha''(t_i)\cdot v_0=\pm\kappa(\alpha(t_i)) \neq0$$
for $1\leq i \leq k$. Therefore, by the implicit function theorem there are functions
$$h_i: U \rightarrow V_i$$
where $U$ is a neighborhood of $v_0$ and $V_i$ is a neighborhood of $t_i$ such that $g(v,h_i(v))=0$ for all $v\in U$.  Thus we have a continuous, one to one correspondence between vectors $v$ in a neighborhood of $v_0$ and points $x$ in a neighborhood of $x_i$ such that $x$ is an extrema with respect to $v$. Then on either side of $v_0$ there is one extrema with respect to that vector in $\alpha(V_i)$ for each $1\leq i \leq k$, which implies that $v_0$ is not locally extrema changing near $x_i$ for $1\leq i \leq k$. By our first argument all extrema changes must occur in neighborhoods of the $x_i$, $1\leq i \leq n$. Therefore the bridge change across $v_0$ is at most $2(n-k)$. (Note $n-k$ is the number of inflection points whose normal is $\pm v$.)

We now show that the normal vector to an $s$-inflection point on $\alpha$ is an extrema changing vector with local multiplicity $2$. Since any extrema changing vector must be the normal or antinormal of an inflection point and the normal of each inflection point changes the extrema count locally by $2$, this will imply that a extrema changing vector with multiplicity $n$ will be the normal or antinormal of $n$ inflection points. 

At an $s$-inflection point, $\alpha(t_i)=x_i$, curvature changes signs. Without loss of generality assume $\kappa(t_i-\varepsilon)<0$ and $\kappa(t_i +\varepsilon)>0$ for all sufficiently small $\varepsilon >0$. Then, in this neighborhood of $t_i$ the Frenet frame rotates clockwise as $t$ increases through parameter in $(t_i-\varepsilon,t_i)$ and counterclockwise as $t$ increases through $(t_i,t_i+\varepsilon)$. Then all points within this neighborhood of $t_i$ have  normal vectors counterclockwise of $N_{x_i}$. Thus in any small neighborhood of $N_{x_i}$, vectors clockwise to $N_{x_i}$ are not the normal vectors to any points $\alpha(t)$ for all $t\in(t-\varepsilon,t+\varepsilon)$ and each vector in this neighborhood counterclockwise of $v$ is the normal vector to exactly two extrema in $(t-\varepsilon,t+\varepsilon)$. Then since the normal vector to each inflection point on the curve corresponds to a local extrema change of two and any bridge changing vector must be the normal vector of some inflection point, the lemma follows.
\end{proof}

\begin{proof}[Proof of Theorem \ref{psbltecurvtors}] 
We begin by noting that the projection that minimized $\hat{\mu}(w,P_{v}(K)) + \nu(P_{v}(K))$ cannot have a $u$-inflection point. Otherwise, we may eliminate the inflection point, $x$, by a planar isotopy that gives $x$ the sign of curvature corresponding to the curvature in some deleted neighborhood of $x$.

Now consider a knot projection that minimizes $\hat{\mu}(w,P_{v}(K)) + \nu(P_{v}(K))$ and contains no cusps. Such a projection is guaranteed to exist by the results of Honma and Saeki \cite{Honma}. Then we may parameterize our knot projection by an immersed curve $\alpha$ with no $u$-inflection points. By the previous lemma each change in number of extrema by $2$ along $\alpha$ has a corresponding inflection point. Consider $v$ such that $\hat{\mu}(v,P(K))$ is minimal among $v \in S^1$. Because $\hat{\mu}(v,P(K))=\hat{\mu}(-v,P(K))$, the maximal number of extrema occurs between $v$ and $-v$. To get to the direction, say $w$, with the maximal number of extrema, we need one inflection point for each pair of extrema. Since we gain the same number of extrema between $v$ and $w$ and between $-v$ and $w$, we get that 
$$\max_{w \in S^1} \hat{\mu}(w,\alpha)-\min_{w \in S^1} \hat{\mu}(w,\alpha)\le \nu(\alpha)$$
Then we have 
\[2\max_{v \in S^1} \mu(v,\alpha) = \max_{v \in S^1} \hat{\mu}(v,\alpha) \leq \min_{K \in [K]} \min_{v \in S^2} \min_ {w \in S^1} (\hat{\mu}(w,P_{v}(K)) + \nu(P_{v}(K)))\]
By the previous theorem, the right side of this equality equals $\frac{(\kappa + \tau)[K]}{\pi}$. Further, minimizing the left side of the inequality over both projections and conformation of $K$ provides $2psb[K]$. Hence, we obtain the desired inequality.
\end{proof} 

The bound on curvature-torsion given in \cite{Honma} and the previous theorem provide the following bound for projective superbridge in terms of bridge index.

\begin{proposition}
For a knot $K$, $psb[K] \leq 2b[K] - 1$. 
\end{proposition}

\begin{proof}
Honma and Saeki \cite{Honma} prove that for every knot $K$, 
\[(\kappa + \tau)[K] - 2\pi \leq 2(\kappa[K] -2\pi)\]
where $\kappa[K]$ denotes the total curvature of $[K]$. By a result of Milnor \cite{Milnorcurves}, we know that $\kappa[K] = 2\pi b[K]$.  Applying this and Theorem \ref{psbltecurvtors} to the above inequality, we get that
\[2\pi (psb[K]) -2\pi \leq 2(2\pi b[K] - 2\pi)\]
and obtain $psb[K] \leq 2b[K] - 1$.
\end{proof} 
 
We now relate the curvature-torsion of a knot to its spiral index. 

\begin{proposition}
For any knot $K$, $(\kappa +\tau)[K] \leq 2\pi sp[K]$.
\end{proposition}

\begin{proof}
Consider a spiral projection $P(K)$ of a conformation of $K$ which realizes $sp[K]$. Then $P(K)$ has no $s$-inflection points and, as in the above proof, we may elimate each $u$-inflection points while remaining in spiral form. Thus, $P(K)$ has 2$sp[K]$ extrema, which implies that 
$$\hat{\mu}(v,P(K)) + \nu(P(K)) =\hat{\mu}(v,P(K))=2sp[K]$$
Minimizing the left side of this equation over all projections and conformations provides the desired equality.
\end{proof}

We now have the tools necessary to characterize all knots $K$ for which \\
$(\kappa + \tau)[K]= 6\pi$, a question posed in \cite{Honma}.

\begin{proposition} 
For a knot $K$, $(\kappa + \tau)[K]= 6\pi$ if and only if $psb[K]=3$. Further, these knots are precisely the 2-bridge knots that are not 2-braid and the 3-spiral knots.
\end{proposition}

\begin{proof}
If $(\kappa + \tau)[K]= 6\pi$ then $psb[K]\le 3$. However, if $psb[K]=2$ then $sp[K]=\beta[K] =2$ and $K$ is a 2-braid knot. However, Milnor has proved that 2-braid knots have curvature-torsion equal to $4\pi$. Therefore, $(\kappa + \tau)[K]= 6\pi$ implies that $psb[K]=3$.

Conversely, assume that $psb[K] = 3$. Then $b[K]$ is either $2$ or $3$. If $b[K]=2$ then Honma and Saeki have proved that $(\kappa + \tau)[K]= 6\pi$ whenever $\beta[K] > 2$. If $b[K]=3$, then $b[K]=psb[K]=sp[K]=3$ and $(\kappa + \tau)[K]\le 6\pi$ since curvature-torsion is no greater than spiral index times $2\pi$. Since $(\kappa+\tau)[K]$ is greater than or equal to $2\pi psb[K]$, $(\kappa + \tau)[K]= 6\pi$.

By the argument above, $psb[K] = 3$ implies that either $b[k]=2$ with $\beta[K]>2$ or $sp[K]=3$ as desired. Further, if $b[K]=2$ then $psb[K]\le 2b[K]-1 = 3$. If $\beta>2$ then $psb[K]\neq 2$ implying that $psb[K]=3$. If $sp[K]=3$ then $psb[K]=3$, since $psb[K]=2$ would imply that $K$ is a 2-braid.
\end{proof}

We have also determined that for a knot $K$ with fewer than $9$ crossings, $2\pi psb[K]=(\kappa + \tau)[K]$. It remains open whether there exists a knot $K$ where $2\pi psb[K]<(\kappa + \tau)[K]$.

\bibliographystyle{plain}
\bibliography{spiralrefs}

\end{document}